\documentclass[12pt]{amsart}
\usepackage[utf8]{inputenc}
\usepackage[usenames, dvipsnames]{color}
\usepackage{mathtools}
\usepackage{amscd}
\usepackage{amsmath}
\usepackage{amssymb}
\usepackage{amsfonts}
\usepackage{amsthm}
\usepackage{enumitem}
\usepackage{soul}
\input xy

\calclayout

\newcommand{\cC}{{\mathcal C}}

\newcommand{\cB}{\mathcal{B}}

\newcommand{\Rep}{\operatorname{Rep}}

\newcommand{\NN}[2]{ N^{ #1 }_{#2,\overline{#2}} }

\theoremstyle{plain}

\numberwithin{equation}{section}

\newtheorem{theorem}{Theorem}[section]
\newtheorem{lemma}[theorem]{Lemma}
\newtheorem{corollary}[theorem]{Corollary}

\theoremstyle{definition}
\newtheorem{definition}[theorem]{Definition}

\theoremstyle{remark}

\theoremstyle{remark}

\newcounter{commentcounter}

\newcounter{todocounter}

\author[Galindo]{C\'esar Galindo}
\address{ Departamento de Matem\'aticas, Universidad de los Andes, Bogot\'a, Colombia}
\email{cn.galindo1116@uniandes.edu.co}
\author[Rowell]{Eric Rowell}
\address{Department of Mathematics, Texas A\&M University, College Station, TX}
\email{rowell@math.tamu.edu}
\author[Wang]{Zhenghan Wang}
\address{Microsoft Research Station Q, and Department of Mathematics, University of California, Santa Barbara, CA}
\email{zhenghwa@microsoft.com}

\begin{document}

\title[Acyclic anyon model]{Acyclic anyon models, thermal anyon error corrections, and braiding universality}

\thanks{ C.G. was partially supported by Fondo de Investigaciones de la Facultad de Ciencias de
la Universidad de los Andes, Convocatoria 2018-2019 para la Financiación de Programas de Investigación, programa “SIMETR\'{I}A $T$ (INVERSION TEMPORAL) EN
CATEGOR\'{I}AS DE FUSI\'{O}N Y MODULARES”,  E.R. was partially funded by NSF grant DMS-1664359, and Z.W. was partially funded by NSF grant DMS-1411212 and FRG-1664351.}
\begin{abstract}
Acyclic anyon models are non-abelian anyon models for which thermal anyon errors can be corrected.  In this note, we characterize acyclic anyon models and raise the question if the restriction to acyclic anyon models is a deficiency of the current protocol or could it be intrinsically related to the computational power of non-abelian anyons.  We also obtain general results on acyclic anyon models and find new acyclic anyon models such as $SO(8)_2$ and untwisted Dijkgraaf-Witten theories of nilpotent finite groups.
\end{abstract}


\date{\today}
\maketitle

\section{Introduction}

In topological quantum computing (TQC), information is encoded in the ground state manifolds of topological phases of matter which are error correction codes.  Therefore, TQC is intrinsically fault-tolerant against local errors.  But at any finite temperature $T>0$, thermal anyon pairs created from the vacuum due to thermal fluctuations can diffuse and braid with computational anyons to cause errors, the so-called thermal anyon errors.  In practice, thermal anyon creations are suppressed by the energy gap $\Delta$ and low temperature $T$ as $~\alpha e^{-\frac{\Delta}{T}}$ for some positive constant $\alpha$, so it might not pose a serious challenge.  But if the suppression by gap and temperature is not enough, then thermal anyon errors could become a serious issue for long quantum computation.  In \cite{acyclic}, the authors found an error correction scheme for acyclic anyon models (called non-cyclic in \cite{acyclic}).  In this paper, we characterize acyclic anyon models as anyon models with nilpotent fusion rules.  We obtain several general results on acyclic anyon models and find many more acyclic anyon models such as $SO(8)_2$, which has Property $F$.

Our characterization of acyclic anyon models raise the question if the restriction to acyclic anyon models is a deficiency in the current protocol or could it be intrinsically related to the computational power of non-abelian anyons.  A triality exists for the computational power of non-abelian anyons as illustrated by the anyon models $SU(2)_k, k=2,3,4$.  The type of anyons in $SU(2)_k$ is labeled by the truncated angular momenta in $\{0,1/2,\ldots,k/2\}$ and let $s$ be the spin=$1/2$ anyon.  When $k=2$, $s$ is essentially the Ising anyon $\sigma$, not only it is not braiding universal, but also all braiding circuits can be efficiently simulated by a Turing machine.  Moreover, it is believed that all measurements of total charges can also be efficiently simulated classically.  When $k=3$, $s$ is the Fibonacci anyon, which is braiding universal \cite{flw02}.  When $k=4$, $s$ is a metaplectic anyon which is not braiding universal.  But supplemented by a total charge measurement, a universal quantum computing model can be designed based on the metaplectic anyon $s$ \cite{cui15}.  While $SU(2)_2$ is acyclic, neither $SU(2)_3$ nor $SU(2)_4$ is.  Since acylic anyon models are weakly integral (proved below), they should not be braiding universal as the property $F$ conjecture suggests \cite{NaiduRowell}.  Therefore, it would be interesting to know if any acylic model can be made universal when supplemented with total charge measurements.  If not, then whether or not the protocols in \cite{acyclic} can be generalized to go beyond acyclic anyon models.  

\section{Preliminaries}

An anyon model is mathematically a unitary modular tensor category---a very difficult and complicated structure \cite{bull17}.  But the fusion rule of an anyon model is completely elementary.  Our main result is a theorem about fusion rules, so we start with the basics of fusion rules to make the characterization of acyclic anyon model self-contained. 

\subsection{Fusion Rules}

A \emph{fusion rule (A,N)} based on the finite set $A$ is a collection of non-negative integers $\{N_{ij}^k\}$ as below, where the elements of $A$ will be called anyon types or particle types or topological charges. The elements in $A$ will be denoted by $x_1,x_2,x_3,\ldots$.  A fusion rule is really the pair $(A,N)$, but in the following we sometimes simply refer to the set $A$ or the set of integers $\{N_{ij}^k\}$ as the fusion rule when no confusion would arise.

For every particle type $x_i$ there exists a unique dual or anti-particle type, that we denote by $\overline{x}_i=x_{\overline{i}}$. There is a trivial or “vacuum” particle type denoted by $1$.

The \textit{fusion rules} can be conveniently organized into formal fusion product and sum of particle types (mathematically such formal product and sum can be made into operations of a fusion algebra where particle types are bases elements of the fusion algebra):

\begin{equation*}
    x_i x_j= \sum_{k}N_{i,j}^{k}x_k
\end{equation*}
where $N_{i,j}^{k}\in \mathbb{Z}^{\geq 0}$. The fusion rules obey the following relations

\begin{enumerate}[label=(\alph*)]
    \item Associativity: $(x_i x_j) x_k= x_i (x_j x_k)$,
    \item The vacuum is the identity for the fusion product, \[x_i 1=x_i =1x_i,\]
    \item The anti-particle type $x_i\mapsto \overline{x_i}=x_{\overline{i}}$ defines an involution, that is, \[\overline{1}=1,\  \ \overline{x_{\overline{i}}}=x_i,\  \ x_{\overline{j}} x_{\overline{i}}= \overline{x_ix_j},\] where \[\overline{x_ix_j}:= \sum_k N_{i,j}^{k} x_{\overline{k}},\]
    \item The fusion of $x_i$ with its antiparticle $x_{\overline{i}}$ contains the vacuum with multiplicity one, that is $$N_{i,\overline{i}}^1=1.$$

\end{enumerate}

A fusion rule is called \emph{abelian} (or pointed) if $$\sum_k N_{i,j}^{k}=1$$ for every $x_i$ and $x_j$.  If $A$ is an abelian fusion rule, then the fusion product defines a group structure on $A$ and conversely every  group defines an abelian fusion rule.



\subsection{Nilpotent fusion rules}
Let $(A,N)$ be a fusion rule on the set $A$.   A \emph{sub-fusion rule} of $(A,N)$ is a subset $B\subset A$ such that 

\begin{enumerate}[label=(\alph*)]
    \item $1 \in B$,
    \item  $x_i\in B$ if and only if $x_{\overline{i}}\in B$,
    \item  if $x_i,x_j\in B$, then $N_{i,j}^{k}>0$ implies $x_k\in B$.
\end{enumerate} The rank of the fusion rule $A$ is $|A|$, the cardinality of the set $A$.

\begin{definition}\cite{GELAKI20081053}
Let $A_{ad}$ be the minimal sub-fusion rule of $A$ with the property that $x_ix_{\overline{i}}$ belongs to $A_{ad}$ for all $x_i \in A$; that is, $A_{ad}$ consists of all particle types of $A$
contained in $x_ix_{\overline{i}}$ for all $x_i \in A$.
\end{definition}

\begin{definition}\cite{GELAKI20081053}
 The \emph{descending central series} of $A$ is the series of sub-fusion rules \[\cdots A^{(n+1)}\subseteq A^{(n)}\subseteq \cdots \subseteq A^{(1)}\subseteq A^{(0)}=A,\] defined recursively as $A^{(n+1)}=A^{(n)}_{ad},$ for all $n\geq 0$.
\end{definition}

\begin{definition}\cite{GELAKI20081053}
A fusion rule is called \emph{nilpotent}, if there exists an $n\in \mathbb{N}$ such that $A^{(n)}$ has rank one. The smallest number $n$ for which this happens is called the \emph{nilpotency class} of $A$.
\end{definition}




\section{Acyclic fusion rules are nilpotent}

In this section, we prove our main result.

\subsection{Acyclic fusion rules}

\begin{definition}\cite{acyclic}
A fusion rule $A$ is called \emph{acyclic} if for any value of $n\in \mathbb{N}$ and for any sequence $$(x_{i_1}=x_{i_{n+1}} , x_{i_n},\ldots ,x_{i_3}, x_{i_2}, x_{i_1})$$ 
with $x_{i_1} \neq \mathbf{1}$, we have that $$\prod_{k=1}^{n} \NN{i_k}{i_{k+1}} =0.$$
\end{definition}

\begin{figure}
    \centering
    \includegraphics[width=10cm]{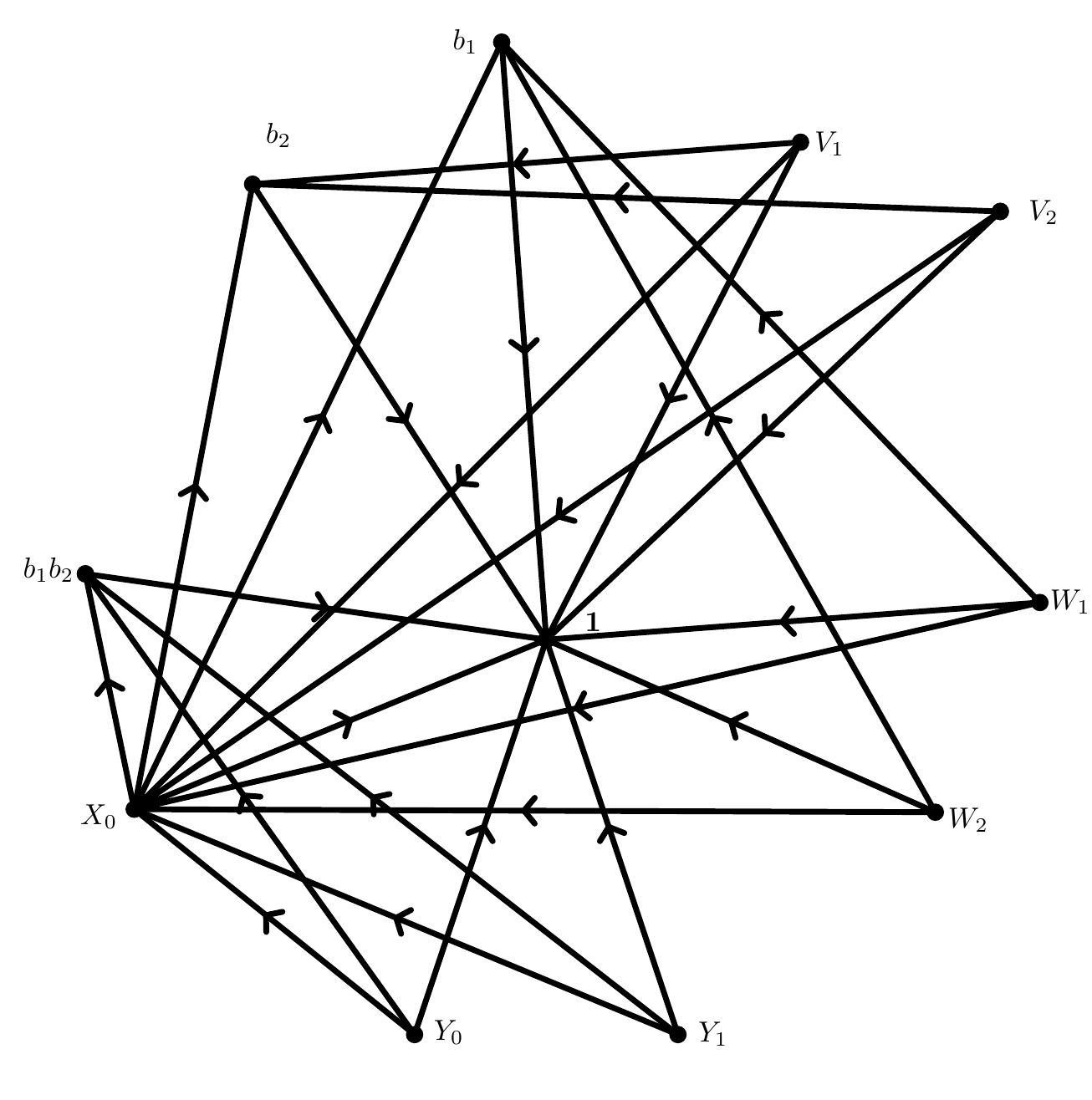}
    \caption{$SO(8)_2$ adjoint graph: $b_1,b_2,b_1b_2$ are bosons and the remaining objects have dimension $2$.  All objects are self-dual. }
    \label{so8level2}
\end{figure}

To any fusion rule we may associate its \emph{adjoint graph} defined as follows (\cite{acyclic}): the vertices are pairs $X_i:=(x_i,x_{\overline{i}})$ and a directed edge is drawn from $X_i\neq (\mathbf{1},\mathbf{1})$ to $X_j$ if $N_{i,\overline{i}}^{j}\neq 0$.  Notice that this is unambiguous since $N_{i,\overline{i}}^{j}=N_{i,\overline{i}}^{\overline{j}}$.  An example is found in Figure \ref{so8level2}.  Now we can alternatively say that a fusion rule is acyclic if its adjoint graph contains no directed cycles.
The adjoint graph found in Figure \ref{so8level2} corresponds to $SO(8)_2$, an integral modular category of dimension $32$ and rank $11$: the explicit fusion rules are found in \cite{BGPR}.  Notice that there are no directed cycles in the adjoint graph of $SO(8)_2$ so its fusion rule is acyclic.  Note also that the direct product of two acyclic fusion rules is acyclic as well.

\begin{lemma}\label{lemma:key-lemma}
Let  $A$ be finite acyclic fusion rule with $|A|>1$. Then the rank of $A_{ad}$ is strictly smaller than the rank of $A$.
\end{lemma}

\begin{proof}
Assume that $A$ is acyclic and $A_{ad}=A$.

For each $n\in \mathbb{N}$ we will define a sequence of bases elements 
\begin{equation}\label{sequence}
(x_{i_n},\ldots,x_{i_2}, x_{i_1})    
\end{equation}
such that 
\begin{enumerate}[label=(\alph*)]
    \item\label{a} $\NN{i_k}{i_{k+1}}>0$ and $\NN{i_{k+1}}{i_k}=0$ for all $k<n$.
\item\label{c} $x_{i_k}\neq 1$ for all $k$.
\end{enumerate}

Since $A$ has rank bigger than one, there is an $x_{i_1}\neq 1$. Using that $A_{ad}=A$, we have that there is $x_{i_2}$ such that $\NN{i_1}{i_2}>0$.  Now, since $A$ is acyclic, using the sequence $(x_{i_1},x_{i_2},x_{i_1})$, we have that $\NN{i_2}{i_1}=0.$ In particular, $\NN{i_2}{i_1}=0$ implies  $x_{i_2}\neq 1$.  Using the same argument, we can construct for each $n\in \mathbb{N}$ a sequence  $(x_{i_n},\ldots, x_{i_1})$  that satisfies \ref{a}  and \ref{c}.

Let us see that the elements in the sequence \eqref{sequence} are pairwise distinct.  For $n=2$,  we have that $\NN{i_1}{i_2}\neq \NN{i_2}{i_1}$, then $x_{i_1}\neq x_{i_2}$.  Assume that any sequence of  $n-1$ elements satisfying \ref{a} and \ref{c} has pairwise distinct elements. Then  $(x_{i_n},\ldots,x_{i_2})$ and  $(x_{i_{n-1}},\ldots, x_{i_1})$ are pairwise distinct. If $x_{i_1}=x_{i_n}$, since $A$ is acyclic, we have that $$\prod_{k=1}^{n}\NN{i_k}{i_{k+1}}=0.$$ But by construction, $\NN{i_k}{i_{k+1}}>0$, hence we have a contradiction. In conclusion, the elements in the sequence  $(x_{i_n},\ldots, x_{i_1})$ are pairwise distinct.

Finally, since the rank of $A$ is a finite number, and we can construct an arbitrary large sequence of pairwise distinct basic elements, we obtain a contradiction. Thus if $A$ is a nontrivial acyclic fusion rules, the rank of $A_{ad}$ is strictly smaller than the rank of $A$. 
\end{proof}

\begin{theorem}\label{Thm:acyclic implies nilpotent}
Let $A$ be a fusion rule. Then $A$ is  acyclic if and only if $A$ is nilpotent.
\end{theorem}
\begin{proof}
Clearly any sub-fusion rules of acyclic fusion rules is acyclic. 

Assume that $A$ is acyclic. Using Lemma \ref{lemma:key-lemma} we obtain that in the sequence \[\cdots A^{(n+1)}\subseteq A^{(n)}\subseteq \cdots \subseteq A^{(1)}\subseteq A^{(0)}=A,\] the rank of $A^{(n+1)}$ is strictly smaller than the rank of $A^{(n)}$ if the rank of $A^{(n)}$ is bigger than one. Since the rank of $A$ is finite, there is $m\in \mathbb{N}$ such that the rank of  $A^{(m)}$ is one, that is, $R$ is nilpotent.

Assume that $A$ is nilpotent. We will use induction on the nilpotency class. If $A$ has nilpotency class one, then $A$ is abelian (pointed in mathematical terminology) and thus acyclic. Assume that that $A$ is nilpotent with $A_{ad}\neq A$. By induction hypothesis $A_{ad}$ is acyclic. Let $$(x_{i_1}=x_{i_{n+1}},x_{i_n},\ldots ,x_{i_2},x_{i_1})$$ be a sequence of basic elements with $x_{i_1}\neq 1$. If $\NN{i_k}{i_{k+1}}>0$ for all $k$, then $x_{i_k}\in A_{ad}$ for all $k$ and $$\prod_{k=1}^n\NN{i_k}{i_{k+1}}>0,$$ a contradiction since $A_{ad}$ is acyclic.
\end{proof}
\begin{corollary}
\begin{enumerate}
    \item\label{gauging} If a gauging  $\cB_G^{\times, G}$ of a modular category $\cB$ by a finite group $G$ has acyclic fusion rules then $\cB$ has acyclic fusion rules and $G$ is nilpotent.
    \item\label{Dijkgraaf-Witten} The untwisted Dijkgraaf-Witten theory $\mathcal{Z}(\operatorname{Vec}_G)$ has acyclic fusion rules if and only if $G$ is a nilpotent group.
\end{enumerate}
\end{corollary}
\begin{proof}
\eqref{gauging} If $\cB_G^{\times, G}$ is nilpotent any fusion subcategory is also nilpotent. Since $\Rep(G)\subset \cB^G\subset \cB_G^{\times, G}$, we have that $\Rep(G)$ and $\cB^G$ are nilpotent. The forgetful functor $\cB^G\to \cB$ is suryective, then by \cite[Proposition 4.6]{GELAKI20081053}


\eqref{Dijkgraaf-Witten} An untwisted Dijkgraaf-Witten theory of a finite group $G$ is exactly $\mathcal{Z}(\Rep(G))$ the Drinfeld center of $\Rep(G)$. Thus, by \cite[Theorem 6.11]{GELAKI20081053} $\mathcal{Z}(\Rep(G))$ is nilpotent if and only if $G$ is nilpotent.

\end{proof}

A braided fusion category $\cC$ has property $F$ \cite{NaiduRowell} if, for every simple object $X$, the braid group representation associated with $X$ has finite image.  Conjecturally, the class of braided fusion categories with property $F$ coincides with the class of braided weakly integral fusion categories.  It follows from \cite{NaiduRowell} that the acyclic braided fusion category $SO(8)_2$ mentioned above has property $F$.  Although we do not know if all acyclic braided fusion categories have property $F$, some partial results in this direction are as follows.

\begin{corollary}
\begin{enumerate}
\item A fusion category with acyclic fusion rules is weakly group-theoretical. In particular it is weakly integral.
\item An integral braided fusion category with acyclic fusion rules is group-theo\-retical and hence property $F$.
\item A modular tensor category $\cB$ has acyclic fusion rules if and only if $\cB$ is the Deligne product of modular categories of prime powers.
    
\end{enumerate}
\end{corollary}

\begin{proof}
In light of Theorem \ref{Thm:acyclic implies nilpotent}, the statements follow from the results of \cite{DGNO1,ERW,ENO2}.
\end{proof}
This strongly suggests that anyon models with acyclic fusion rules are never braiding universal alone.

\end{document}